\documentclass[12pt]{amsart}

\usepackage{amscd,mathrsfs,epsfig,amsthm,amssymb,amsmath}
\usepackage[all]{xy} 
\usepackage[T1]{CJKutf8}
\usepackage[sans]{dsfont}  
\usepackage{hyperref}
\usepackage{tabularx}
\usepackage{rotating}

\usepackage{graphicx} 
\usepackage{cite}
\usepackage{url}

\newtheorem{theorem}{Theorem}[subsection]
\newtheorem{corollary}[theorem]{Corollary}
\newtheorem{lemma}[theorem]{Lemma}
\newtheorem{definition}[theorem]{Definition}

\newtheorem{remark}[theorem]{Remark}

\newtheorem{example}[theorem]{Example}

\newtheorem*{thma}{Theorem A}
\newtheorem*{thmb}{Theorem B}
\newtheorem*{thmc}{Theorem C}

\allowdisplaybreaks[2]

\def\calC{{\mathcal C}}

\def\End{\mathop{\rm End}\nolimits}

\def\Hom{\mathop{\rm Hom}\nolimits}

\def\Id{\mathop{\rm Id}\nolimits}

\def\lim{\mathop{\varinjlim}\nolimits}
 
\def\Ob{\mathop{\rm Ob}\nolimits} 
\def\Mor{\mathop{\rm Mor}\nolimits}

\def\cod{\mathop{\rm cod}}
\def\dom{\mathop{\rm dom}}

\DeclareMathOperator{\Ab}{\rm Ab}
\DeclareMathOperator{\rMod}{Mod-}

\DeclareMathOperator{\Add}{\rm Add}
\DeclareMathOperator{\op}{\rm op}

\begin{document}

\title[Pseudoskew category algebras]{Pseudoskew category algebras and modules over representations of small categories}

\author{Mawei Wu}
\address{School of Mathematics and Statistics, Lingnan Normal University, Zhanjiang, Guangdong 524048, China}
\email{wumawei@lingnan.edu.cn}

\thanks{The author \begin{CJK*}{UTF8}{}
\CJKtilde \CJKfamily{gbsn}(吴马威)
\end{CJK*} is supported by the Scientific Research Innovation Project of Lingnan Normal University (No. LT2401) and Lingnan Normal University (No. LT2410)}





\begin{abstract}
Let $\calC$ be a small category and let $R$ be a representation of the category $\calC$, that is, a pseudofunctor from a small category to the category of small preadditive categories. In this paper, we mainly study the category $\rMod R$ of right modules over $R$. We characterize it both as a category of the Abelian group valued functors on $Gr(R)$ and as a category of modules over a new family of algebras: the pseudoskew category algebras $R[\calC]$, where $Gr(R)$ is the linear Grothendieck construction of $R$. Moreover, we also classify the hereditary torsion pairs in $\rMod R$ and reprove a result of Estrada and Virili.     
\end{abstract}

\maketitle

\tableofcontents

\section{Introduction}
Let $\calC$ be a small category and let $R: \calC \to \Add$ be \emph{a representation of the category} $\calC$ which is a pseudofunctor from a small category $\calC$ to the category of small preadditive categories $\Add$ (see Definition  \ref{rep} for its precise definition). Given a representation $R$ of the category $\calC$, one can consider \emph{the category $\rMod R$ of right modules over $R$} (see Definition \ref{rmod}). In 2017, Estrada and Virili showed that the category of right $R$-modules $\rMod R$ is a Grothendieck category, and if $\calC$ is a poset, it also has a projective generator (see \cite[Theorem 3.18]{EV17}). The study of the representation of the category follows the philosophy of Mitchell \cite{Mit72} of working with rings with several objects. With the same spirit, one can replace the preadditive categories $\Add$ here by other categories. For instance, one can take values in entwining structures over a semiperfect coalgebra \cite{Ban23}, or taking values in (co)algebras \cite{BBR24}. Recently, we consider replacing the preadditive categories by the differential graded categories \cite{Wu24a}. 

In this paper, we change slightly the definition of $R$-modules in \cite[Definition 3.6]{EV17} (see Remark \ref{change}), and trying to investigate the category $\rMod R$ further. We obtain two characterizations of it, both as a functor category and as a category of modules over an algebra. Given a representation $R$ of the category, one can define the linear Grothendieck construction $Gr(R)$ of it (see Definition \ref{lingrocon}). Our first result characterizes the category $\rMod R$ as a functor category, and reproves (actually our result is slightly more general) Estrada and Virili's result \cite[Theorem 3.18]{EV17} with different method.

\begin{thma} (Theorem \ref{howeplus} and Corollary \ref{Groproj})
Let $\calC$ be a small category and let $R: \calC \to \Add$ be a representation of the category $\calC$, then we have the following equivalence
 $$
 \rMod R \simeq (Gr(R)^{\op}, \Ab).
 $$ 
Consequently, the category of right $R$-modules $\rMod R$ is a Grothendieck category and has a projective generator.     
\end{thma}

The first part of Theorem A above can be viewed as the pseudofunctor analogue of the Howe's result in trivial topology case (see \cite[Proposition 5]{How81} or Theorem \ref{How81}). With the help of Theorem A, we can classify all the hereditary torsion pairs in $\rMod R$ by the linear Grothendieck topologies on $Gr(R)$. For the definitions of hereditary torsion pairs and linear Grothendieck topologies, one can see \cite[Definition 2.3.1 \& 2.3.2 and Definition 2.1.15]{Wu24}.

\begin{thmb} (Corollary \ref{htp})   
Let $\calC$ be a small category and let $R: \calC \to \Add$ be a representation of the category $\calC$. Then there is an (explicit) one-to-one correspondence between linear Grothendieck topologies on $Gr(R)$ and hereditary torsion pairs in $\rMod R$. 
\end{thmb}

Given a representation $R$ of a small category $\calC$, we introduce a new family of algebras associated to it, so-called the \emph{pseudoskew category algebras} (see Definition \ref{pseudoskewcatealg}). The pseudoskew category algebras include skew category algebras (see \cite[Definition 3.2.1]{WX23}) as special cases. Our second characterization of $\rMod R$ is as follows, which says that $\rMod R$ is equivalent to a category of modules over a pseudoskew category algebra $R[\calC]$. The following Theorem can be seen as a higher analogue of the \cite[Theorem A]{WX23}.

\begin{thmc} (Theorem \ref{higherskew})
Let $\calC$ be a small category and let $R: \calC \to \Add$ be a representation of the category $\calC$. If $\Ob \calC < +\infty$ and $\Ob R(i) < +\infty$ for all $i \in \Ob \calC$, then we have the following equivalence
 $$
 \rMod R \simeq \rMod R[\calC].
 $$    
\end{thmc}

This paper is organized as follows. In Section \ref{prelim}, the definitions of the representation of a small category and its right modules, as well as the linear Grothendieck constructions are recalled. Then, a new notion, so-called the pseudoskew category algebra, is introduced. In Section \ref{cha}, two characterizations of categories of modules $\rMod R$ over representations of a small category $\calC$ are given, both as an Abelian group valued functor category and as a category of modules over a pseudoskew category algebra. Besides, the hereditary torsion pairs on $\rMod R$ are classified and a result (\cite[Theorem 3.18]{EV17}) of Estrada and Virili is reproved in this Section.

\section{Preliminaries} \label{prelim}
In this section, we will recall the definitions of the representation of a small category and its right modules, as well as the definition of the linear Grothendieck constructions. And then, a new concept, so-called the pseudoskew category algebra, will be introduced.

\subsection{Modules over representations of small categories}

In this subsection, the definitions of a representation of a small category and its right modules will be recorded.

\begin{definition}(\cite[Definition 3.1]{EV17}) \label{rep}
Let $\calC$ be a small category, a \emph{representation} of $\calC$ is a pseudofunctor $R: \calC \to \Add$, that is, $R$ consists of the following data:

\begin{enumerate}
    \item for each object $i \in \Ob \calC$, a preadditive category $R(i)$;
    \item for all $i, j \in \Ob \calC$ and any morphism $a: i \to j$, an additive functor $R(a): R(i) \to R(j)$;
    \item for each object $i \in \Ob \calC$, an isomorphism of functors $\delta_i: 1_{R(i)} \overset{\sim}{\longrightarrow} R(1_i)$;
    \item for any pair of composable morphisms $a$ and $b$ in $\calC$, an isomorphism of functors $\mu_{b,a}: R(b)R(a) \overset{\sim}{\longrightarrow} R(ba)$.    
\end{enumerate}
Furthermore, we suppose that the following axioms hold:

\begin{enumerate}
    \item[(Rep.1)] given three composable morphisms $i \overset{a}{\longrightarrow} j \overset{b}{\longrightarrow} k \overset{c}{\longrightarrow} h$ in $\calC$, the following diagram commutes
     $$
     \xymatrix @C=5pc {
      R(c)R(b)R(a) \ar[rr]^{R(c)\mu_{b,a}} \ar[d]_{\mu_{c,b} R(a)} & & R(c)R(ba) \ar[d]^{\mu_{c,ba}} \\
      R(cb)R(a) \ar[rr]^{\mu_{cb,a}}  & & R(cba), \\
     } 
     $$ 
    \item[(Rep.2)] given a morphism $(a: i \to j) \in \Mor \calC$, the following diagram commutes
     $$
     \xymatrix{
        & R(a) \ar[dl]_{R(a)\delta_i} \ar[dr]^{\delta_jR(a)} \ar@{=}[dd] & \\
   R(a)R(1_i) \ar[dr]_{\mu_{a,1_i}}     &  &  R(1_j)R(a) \ar[dl]^{\mu_{1_j,a}}  \\
      & R(a). &  \\
     } 
     $$ 
\end{enumerate}
\end{definition}

\begin{remark} \label{diff}
\begin{enumerate}
    \item One should note that the 2-isomorphisms $\delta$ and $\mu$ above are different  from the 2-isomorphisms (when restricting a colax functor to a pseudofunctor) $\eta$ and $\theta$ in \cite[Definition 2.1]{Asa13b}. In fact, they are in different directions, 
    $$
     \xymatrix @C=5pc {
     R(1_i) \ar@<.6ex>[r]^{\eta_i}  &  1_{R(i)} \ar@<.6ex>[l]^{\delta_i},  \\
     R(ba)  \ar@<.6ex>[r]^{\theta_{b,a}} &  R(b)R(a) \ar@<.6ex>[l]^{\mu_{b,a}}.   \\
     } 
     $$
     And they are \emph{inverse} to each other, namely, $\delta_i=(\eta_i)^{-1}$ and $\mu_{b,a}=(\theta_{b,a})^{-1}$.  
     \item A representation $R: \calC \to \Add$ is said to be \emph{strict} if it is a functor, that is, $R(1_i)=1_{R(i)}$, $R(ba)=R(b)R(a)$, and $\eta$  and $\mu$ are identities, so do $\delta$ and $\mu$. 
     \item Given a representation $R: \calC \to \Add$ and a morphism $(a: i \to j) \in \Mor \calC$, we denote by 
     $$
     \xymatrix @C=5pc { 
     a_{!}: (R(i)^{\op}, \Ab) \ar@<.6ex>[r] & (R(j)^{\op}, \Ab) :a^{*} \ar@<.6ex>[l] \\
     }
     $$
the change of base adjunction $(a_{!}, a^{*})$ induced by $R(a)$.
\end{enumerate}  
\end{remark}

Given a representation of a small category, one can consider its right modules.

\begin{definition}(\cite[Definition 3.6 with slight modifications]{EV17}) \label{rmod}
Let $R: \calC \to \Add$ be a representation of the small category $\calC$. A \emph{right $R$-module} $M$ consists of the following data:

\begin{enumerate}
    \item for all $i \in \Ob \calC$, a right $R(i)$-module $M_i: R(i)^{\op} \to \Ab$;
    \item for any morphism $a: i \to j$ in $\calC$, a homomorphism $M(a): a^*M_j \to M_i$. 
\end{enumerate}
Furthermore, we suppose that the following axioms hold:

\begin{enumerate}
    \item[(Mod.1)] given two morphisms $a: i \to j$ and $b: j \to k$ in $\calC$, the following diagram commutes:
     $$
     \xymatrix @C=5pc {
      a^*b^*M_k \ar[d]_{\theta_{b,a}1_{M_k}} \ar[r]^{a^*M(b)} & a^*M_j \ar[r]^{M(a)} & M_i \\
     (ba)^*M_k \ar[urr]_{M(ba)} & & \\
     } 
     $$ 
    \item[(Mod.2)] for all $i \in \Ob \calC$, the following diagram commutes:
    $$
    \xymatrix @C=5pc {
      (1_i)^*M_i \ar[r]^-{M(1_i)}  & M_i \\
      M_i \ar[u]^{\eta_i 1_{M_i}} \ar[ur]_{1_{M_i}}  & \\
     } 
    $$
\end{enumerate}
\end{definition}

For a representation $R: \calC \to \Add$ of the small category $\calC$, we will denote \emph{the category of all right $R$-modules} by $\rMod R$, which is the main object studied in this paper.

\begin{remark} \label{change}
The condition (2) of the Definition above is the main difference between ours with \cite[Definition 3.6]{EV17}. We change the direction in order to adapt to the direction chosen in Definition \ref{lingrocon}.    
\end{remark}

\subsection{Linear Grothendieck constructions and pseudoskew category algebras}

In this subsection, we will first recall the definition of \emph{linear} Grothendieck constructions, and then we will introduce a new notion, so-called pseudoskew category algebras. 

\subsubsection{linear Grothendieck constructions}

Given an oplax functor, one can define the linear Grothendieck construction of it, see \cite[Definition 4.1]{Asa13b} for its explicit definition. We can restrict this construction to a pseudofunctor. More specifically, we will apply the linear Grothendieck construction to a representation $R:\calC \to \Add$ of a small category $\calC$. Hence we have the following definition. We already knew that the 2-isomorphisms of a pseudofunctor in \cite[Definition 3.1]{EV17} and \cite[Definition 2.1]{Asa13b} are different, but they are inverse to each other, see Remark \ref{diff} (1).

\begin{definition} \label{lingrocon}
 Let $\calC$ be a small category and let $R:\calC \to \Add$ be a representation of $\calC$. Then a category $Gr(R)$, called the \emph{linear Grothendieck construction} of $R$, is defined as follows:
 \begin{enumerate}
     \item $\Ob Gr(R):=\cup_{i \in \Ob \calC}\{i\} \times \Ob R(i)=\{~_ix:=(i,x) ~|~ i \in \Ob \calC, x \in \Ob R(i)\}$;
     \item for each $_ix, _jy \in \Ob Gr(R)$, we set
     $$
     Gr(R)(_ix, _jy):=\bigoplus_{a \in \calC(i,j)}R(j)(R(a)x,y);
     $$
     \item for each $_ix, _jy, _kz \in \Ob Gr(R)$ and each $f=(f_a)_{a \in \calC(i,j)}\in Gr(R)(_ix, _jy)$, $g=(g_b)_{b \in \calC(j,k)}\in Gr(R)(_jy, _kz)$, we set
     $$
     g \circ f:=\left( \sum_{\substack{a \in \calC(i, j)\\ b \in \calC(j, k)\\ c=ba}} g_b \circ R(b)f_a \circ \theta_{b,a}x \right)_{c \in \calC(i, k)}
     $$
     where each summand is the composite of 
     $$
    \xymatrix @C=3pc {
    R(ba)x \ar[r]^-{\theta_{b,a}x} & R(b)R(a)x \ar[r]^-{R(b)f_a} &  R(b)y \ar[r]^-{g_b} & z;\\
     } 
    $$
     \item for each $_ix \in \Ob Gr(R)$ the identity $1_{_ix}$ is given by
     $$
     1_{_ix}:=(\delta_{a,1_i} \eta_ix)_{a \in \calC(i, i)} \in \bigoplus_{a \in \calC(i,i)}R(i)(R(a)x,x),
     $$
     where $\eta_ix: R(1_i)x \to 1_{R(i)}x=x$, and $\delta_{a,1_i}$ is the Kronecker delta (\emph{not} a 2-isomorphism $\delta_i$ in a pseudofunctor!), that is, the $a$-th component of $1_{_ix}$ is $\eta_ix$ if $a=1_i$, and $0$ otherwise. 
 \end{enumerate}
\end{definition}

\subsubsection{Pseudoskew category algebras}

In this subsection, we will introduce a new notion: pseudoskew category algebras, which will be use to characterize the categories of modules over representations of a small category in Section \ref{chaalg}. 

\begin{definition}  \label{pseudoskewcatealg}
Let $\calC$ be a (non-empty) small category and let $R: \calC \to \Add$ be a representation of $\calC$. The \emph{pseudoskew category algebra} $R[\calC]$ on $\calC$ with respect to $R$ is a $\mathbb{Z}$-module spanned over elements of  
$$
\{f_a ~|~ a \in \calC(i,j), f: R(a)(x) \to y, x \in \Ob R(i), y \in \Ob R(j) \}.
$$
We define the multiplication on two base elements by the rule
	\begin{eqnarray}
		g_b \ast f_a =
		\begin{cases}
		  (g \circ R(b)f \circ \theta_{b,a}x)_{ba},        & \text{if}\ \dom(b)=\cod(a); \notag \\
			0, & {\rm otherwise}.
		\end{cases}
	\end{eqnarray} 
where the map $g \circ R(b)f \circ \theta_{b,a}x$ can be depicted as follows
    $$
    \xymatrix @C=3pc {
    R(ba)x \ar[r]^-{\theta_{b,a}x} & R(b)R(a)x \ar[r]^-{R(b)f} &  R(b)y \ar[r]^-{g} & z,\\
     } 
    $$
and $\theta_{b,a}: R(ba) \overset{\sim}{\longrightarrow} R(b)R(a)$ is a 2-isomorphism of $R$. Extending this product linearly to two arbitrary elements, $R[\calC]$ becomes an associative algebra.
\end{definition}

\begin{remark}
\begin{enumerate}
    \item If $\Ob \calC < +\infty$ and $\Ob R(i) < +\infty$ for all $i \in \Ob \calC$, then the pseudoskew category algebra has an identity 
    $$
    \sum_{_ix \in \Ob Gr(R)}1_{_ix}=\sum_{_ix \in \Ob Gr(R)}(\cdots,0,\eta_ix,0,\cdots),
    $$
    where $\eta_ix: R(1_i)x \overset{\sim}{\longrightarrow} 1_{R(i)}x$ is the $1_i$-th component of $(\cdots,0,\eta_ix,0,\cdots) \in Gr(R)(_ix,_ix)=\bigoplus_{a \in \calC(i,i)}R(i)(R(a)x,x)$;
    \item When $R: \calC \to {\rm Ring}$ is a precosheaf of rings, thus $R$ is a functor, which can be viewed as a pseudofuntor with 2-morphisms are identities (see Remark \ref{diff}), then $R[\calC]$ is just a (covariant version) skew category algebra (see \cite[Definition 3.2.1]{WX23}); 
    \item When $R$ is a lax functor (\emph{not} just a pseudofunctor), then one can define a more general notion, so-called the \emph{lax skew category algebra}.
\end{enumerate}   
\end{remark}

\section{Two characterizations of categories of modules over representations of a small category} \label{cha}

In this section, two characterizations of categories of modules $\rMod R$ over representations of a small category will be given. More specifically, we will characterize the category $\rMod R$ both as an Abelian group valued functor category and as a category of modules over a pseudoskew category algebra. 

\subsection{Characterizing it as functor categories}
Let's first consider the following two constructions, they will be use to prove the main results later.

\noindent\textbf{Construction 1}: Given $M \in \rMod R$, let's define $F_M$ (we will show that $F_M \in (Gr(R)^{\op}, \Ab)$ in Lemma \ref{F}) as follows:
$$
\xymatrix{
    Gr(R)^{\op} \ar[rr]^{F_M} &  & \Ab &  \\
   _ix  \ar@{}[u]|{\begin{sideways}$\in$\end{sideways}} \ar[dd]_{(\cdots,0,f_a,0,\cdots)} & & M_i(x) \ar@{}[u]|{\begin{turn}{90}$\in$\end{turn}} & \\
     & \longmapsto & & M_j(R(a)x) \ar[ul]_{M(a)_x} \\
   _jy  & & M_j(y) \ar[uu]|{M(a)_x \circ M_j(f)} \ar[ur]_{M_j(f)} & \\
     } 
$$
For a general morphism $(f_a)_a \in \Hom_{Gr(R)}(_ix,_jy)$, we define
$$
F_M[(f_a)_a]:=\sum_{a,f}M(a)_x \circ M_j(f).
$$

\noindent\textbf{Construction 2}: Given $F \in (Gr(R)^{\op}, \Ab)$, let's define $M_F$ (we will show that $M_F \in \rMod R$ in Lemma \ref{M}) as follows:
$$
\xymatrix{
   i \in \Ob \calC & \overset{M_F}{\longmapsto} & (R(i)^{\op} \ar[rr]^{M_i} & & \Ab) \\
   &  & x \ar@{}[u]|{\begin{sideways}$\in$\end{sideways}} \ar[dd]_{g} & & F(_ix) \ar@{}[u]|{\begin{sideways}$\in$\end{sideways}} \\
   & & & \longmapsto  & \\
   & & x' & & F(_ix') \ar[uu]_{F[(\cdots,0,g \circ \eta_ix,0,\cdots)]} \\
     } 
$$
where $g \circ \eta_ix$ is the composite of 
$$
R(1_i)x \overset{\eta_ix}{\longrightarrow} 1_{R(i)}x=x \overset{g}{\longrightarrow} x'.
$$
and $(\cdots,0,g \circ \eta_ix,0,\cdots)$ is a morphism of $\Hom_{Gr(R)}(_ix,_ix')$.
Recall that $a^*: (R(j)^{\op}, \Ab) \to (R(i)^{\op}, \Ab)$ is the restriction functor along $R(a): R(i) \to R(j)$, we define $M(a): a^*M_j \to M_i$ as follows:
$$
\xymatrix @C=8pc {
   a^*M_j(x) \ar[r]^-{M(a)_x} \ar @{=} [d] & M_i(x) \ar @{=} [dd] \\
   M_j(R(a)x) \ar @{=} [d] &  \\
  F(_jR(a)x) \ar[r]^-{F[(\cdots,0,1_{R(a)x},0,\cdots)]}  & F(_ix) \\
     } 
$$
where $(\cdots,0,1_{R(a)x},0,\cdots)$ is a morphism in $\Hom_{Gr(R)}(_ix, _jR(a)x)$. Thus we define $M(a)_x:=F[(\cdots,0,1_{R(a)x},0,\cdots)]$, where $1_{R(a)x}$ is in the $a$-th component.

\begin{lemma} \label{F}
If $M \in \rMod R$, then $F_M \in (Gr(R)^{\op}, \Ab)$, where $F_M$ is defined in Construction 1.    
\end{lemma}

\begin{proof}
We will prove that $F_M$ is a contravariant functor. Since 
\begin{align*}
     & F_M(1_{_ix}) \\
	   = & F_M[(\cdots,0,\eta_ix,0,\cdots)] \\
    def.~of~F_M = & M(1_i)_x \circ M_i(\eta_ix) \\
    (Mod.2)~of~def. \ref{rmod}= & 1_{M_i(x)} \\
    def.~of~F_M = & 1_{F_M(_ix)},\\ 	  
\end{align*}
hence $F_M$ preserves the identity: $F_M(1_{_ix})=1_{F_M(_ix)}$. For any two composable morphisms in $Gr(R)$:
$$
\xymatrix @C=4pc {
 _ix \ar[r]^-{f=(f_a)_a} & _jy \ar[r]^-{g=(g_b)_b} & _kz, \\
     } 
$$
we have to check that $F_M(gf)=F_M(f)F_M(g)$. Let $h:=g_b \circ R(b)f_a \circ \theta_{b,a}x$ and $c:=ba$, we have
\begin{align*}
     & F_M(gf) \\
	   = & F_M((g_b)_b \circ (f_a)_a) \\
	   = & F_M\left[\left(\sum_{c=ba}g_b \circ R(b)f_a \circ \theta_{b,a}x\right)_c\right] \\ 
        def.~of~F_M  = & \sum_{c,h} M(c)_x \circ M_k(\sum_{c=ba}h) \\
        M_k~is~an~add.~functor=& \sum_{c,h} M(c)_x \circ M_k(h) \\
          \overset{(\dagger)}{=} & \sum_{a,b,f,g}M(a)_x \circ M(b)_{R(a)x} \circ M_k[R(b)(f_a)] \circ M_k(g_b) \\
          = & \sum_{a,b,f,g}M(a)_x \circ (M(b)_{R(a)x} \circ M_k[R(b)(f_a)]) \circ M_k(g_b)\\
          \overset{(\ddagger)}{=} & \sum_{a,b,f,g}M(a)_x \circ (M_j(f_a) \circ M(b)_y) \circ M_k(g_b)\\
          = & \left(\sum_{a,f}M(a)_x \circ M_j(f_a)\right) \circ \left(\sum_{b,g}M(b)_y \circ M_k(g_b)\right)\\
          = & F_M[(f_a)_a] \circ F_M[(g_b)_b] \\
          = & F_M(f)F_M(g). \\
\end{align*}

The equality $(\dagger)$ holds because the following diagram commutes (The left square commutes as $M_k$ is a functor, and the right square commutes because of the condition (Mod.1) of Definition \ref{rmod}):
$$
\xymatrix @C=5pc {
M_k(R(b)y) \ar[r]^-{M_k(R(b)f_a)} & M_k(R(b)R(a)x) \ar @{=} [d] \ar[r]^-{M(b)_{R(a)x}}  &  M_j(R(a)x) \ar[ddd]^{M(a)_x} \\
                &    (a^*b^*M_k)(x)  \ar[d]^{M_k(\theta_{b,a}x)}_{\simeq} &\\ 
                &        (ba)^*M_k(x)  \ar @{=} [d]  &              \\
M_k(z) \ar[uuu]^{M_k(g_b)} \ar[r]^{M_k(h)}  & M_k(R(ba)x) \ar[r]^{M(c)_x} &
M_i(x).  \\
     } 
$$

The equality $(\ddagger)$ holds since one can show 
$$
M_j(f_a) \circ M(b)_y=M(b)_{R(a)x} \circ M_k[R(b)(f_a)].
$$
This is the case because we have the following commuting diagram:
$$
\xymatrix @C=4.4pc {  
M_i(x)   &    & &  \\
          &        M_j(R(a)x) \ar[ul]|{M(a)_x} & & \\
M_j(y) \ar[uu]|{M(a)_x \circ M_j(f_a)} \ar[ur]|{M_j(f_a)}   \ar@{}[rr]|(.4){(\flat)}  &   &   M_k(R(b)R(a)x) \ar[ul]|{M(b)_{R(a)x}} \ar[r]|-{M_k(\theta_{b,a}x)} & M_k(R(ba)x). \ar[llluu]_{M(ba)_x} \ar@{}[llu]|(.6){(\natural)} \ar@{}[lld]|(.6){(\sharp)}  \\
           &       M_k(R(b)y) \ar[ul]|{M(b)_y}  \ar[ur]|{M_k[R(b)(f_a)]} & & \\
M_k(z) \ar[uu]|{M(b)_y \circ M_k(g_b)} \ar[ur]|{M_k(g_b)} \ar[rrruu]_{M_k(h)} &           &  & \\
     } 
$$
Due to
\begin{align*}
     & F_M[(\cdots,0,f_a,0,\cdots)] \circ F_M[(\cdots,0,g_b,0,\cdots)] \\
	   = & F_M[(\cdots,0,g_b,0,\cdots) \circ (\cdots,0,f_a,0,\cdots)] \\
\end{align*}
and the definition of $F_M$, the outer triangle commutes. By the condition (Mod.1) of Definition \ref{rmod}, the diagram $(\natural)$ commutes, and the diagram $(\sharp)$ commutes as $M_k$ is a functor. Therefore the diagram $(\flat)$ commutes, then the equality $(\ddagger)$ holds. This completes the proof.
\end{proof}

\begin{lemma} \label{M}
If $F \in (Gr(R)^{\op}, \Ab)$, then $M_F \in \rMod R$, where $M_F$ is defined in Construction 2.    
\end{lemma}

\begin{proof}
 Let's first check that $M_i$ in Construction 2 is a functor. For any $x \in \Ob R(i)$, we have
\begin{align*}
     & M_i(1_x) \\
	   def.~of~M_i= & F[(\cdots,0,\eta_ix,0,\cdots)] \\
      = & F(1_{_ix}) \\
      F~preserves~identity= & 1_{F(_ix)}\\
      def.~of~M_i= & 1_{M_i(x).} \\
\end{align*}
For any $x \overset{f}{\longrightarrow} x' \overset{g}{\longrightarrow} x''$ in $\Mor R(i)$, we have to check $M_i(gf)=M_i(f) \circ M_i(g)$. By the definition of $M_i$ in Construction 2, we have 
\begin{align*}
     & M_i(f) \circ M_i(g) \\
	   = & F[(\cdots,0,f \circ \eta_ix,0,\cdots)] \circ F[(\cdots,0,g \circ \eta_ix',0,\cdots)] \\
      = & F[(\cdots,0,g \circ \eta_ix',0,\cdots) \circ (\cdots,0,f \circ \eta_ix,0,\cdots)] \\
\end{align*}
and $M_i(gf)=F[(\cdots,0,gf \circ \eta_ix,0,\cdots)]$, so to check $M_i(gf)=M_i(f) \circ M_i(g)$, it is enough to show 
$$
(\cdots,0,g \circ \eta_ix',0,\cdots) \circ (\cdots,0,f \circ \eta_ix,0,\cdots)=(\cdots,0,gf \circ \eta_ix,0,\cdots).
$$ 
This is true because the outer morphisms of the following diagram commutes:
$$
\xymatrix @C=5pc {  
R(1_i1_i)x \ar@{=}[d] \ar@{=}[dr] \ar[r]^{\theta_{1_i,1_i}x} &  R(1_i)R(1_i)x \ar@{}[ld]|(.3){(\triangle)} \ar[r]^{R(1_i)(f \circ \eta_ix)} \ar[d]^{R(1_i)\eta_ix} &  R(1_i)x' \ar[r]^{g \circ \eta_ix'} \ar[d]^{\eta_ix'} & x'' \\
R(1_i)x \ar[dr]_{\eta_ix}     &  R(1_i)x \ar@{}[r]|{(\diamondsuit)} \ar[d]^{\eta_ix} \ar[ur]_{R(1_i)(f)} &1_{R(i)}x'=x' \ar[ur]_{g} &  \\
& 1_{R(i)}x=x \ar[ur]_{f} \ar@/_3pc/[uurr]_{gf} & & \\
     } 
$$
The triangle $(\triangle)$ commutes by the inverse of the condition of (Rep.2) of Definition \ref{rep}. Since $\eta_i$ is a natural transformation, hence the parallelogram $(\diamondsuit)$ in the middle commutes. It is easy to see that the other diagrams commute.

In order to check $M_F \in \rMod R$, by the Definition of $R$-modules, we also have to check that the conditions (Mod.1) and (Mod.2) of the Definition \ref{rmod} are both satisfied. Let $i \overset{a}{\longrightarrow} j \overset{b}{\longrightarrow} k$ be two composable morphisms in $\calC$. In order to check the condition (Mod.1), we have to check the following diagram commutes:
     $$
     \xymatrix @C=5pc {
      a^*b^*M_k \ar[d]_{\theta_{b,a}1_{M_k}} \ar[r]^{a^*M(b)} & a^*M_j \ar[r]^{M(a)} & M_i \\
     (ba)^*M_k \ar[urr]_{M(ba)} & & \\
     } 
     $$ 
By the definition of $M_i, M_j, M_k$ in Construction 2, it is equivalent to check the following diagram commutes:
     $$
     \xymatrix @C=7pc {
      F(_kR(b)R(a)x) \ar[d]_{F[(\cdots,0,\theta_{b,a}x \circ \eta_kR(ba)x,0,\cdots)]} \ar[r]^-{F[(\cdots,0,1_{R(b)R(a)x},0,\cdots)]} & F(_jR(a)x) \ar[r]^-{F[(\cdots,0,1_{R(a)x},0,\cdots)]} & F(_ix) \\
     F(_kR(ba)x) \ar[urr]_{F[(\cdots,0,1_{R(ba)x},0,\cdots)]} & & \\
     } 
     $$ 
where $\theta_{b,a}x \circ \eta_kR(ba)x$ is the composite of 
     $$
     \xymatrix @C=3pc {
     R(1_k)R(ba)x \ar[r]^-{\eta_kR(ba)x} & 1_{R(k)}R(ba)x=R(ba)x \ar[r]^-{\theta_{b,a}x} & R(b)R(a)x. \\
     } 
     $$ 
Since $F$ is a functor, it is equivalent to check the following diagram commutes:  
     $$
     \xymatrix @C=7pc {
      _kR(b)R(a)x  & _jR(a)x \ar[l]_-{(\cdots,0,1_{R(b)R(a)x},0,\cdots)}  &  _ix \ar[dll]^{(\cdots,0,1_{R(ba)x},0,\cdots)} \ar[l]_-{(\cdots,0,1_{R(a)x},0,\cdots)} \\
     _kR(ba)x \ar[u]^{(\cdots,0,\theta_{b,a}x \circ \eta_kR(ba)x,0,\cdots)} & & \\
     } 
     $$ 
This is true due to the following commuting diagram (using the inverse of the condition (Rep.2) of the Definition \ref{rep}):  
    $$
     \xymatrix @C=4pc {
     R(ba)x \ar@{=}[d] \ar[r]^{\theta_{b,a}x}  & R(b)R(a)x \ar[r]^{R(b)1_{R(a)x}} & R(b)R(a)x \ar[r]^{1_{R(b)R(a)x}} & R(b)R(a)x\\
     R(ba)x \ar[r]^{\theta_{1_k,ba}x}  & R(1_k)R(ba)x \ar[r]^-{R(1_k)(1_{R(ba)x})} & R(1_k)R(ba)x \ar[r]^-{\theta_{b,a}x \circ \eta_kR(ba)x} & R(b)R(a)x. \ar@{=}[u] \\
     } 
     $$ 
Therefore, the condition (Mod.1) of the Definition \ref{rmod} is satisfied.

In order to check condition (Mod.2), we have to check the following diagram commutes:   
    $$
    \xymatrix @C=5pc {
      (1_i)^*M_i \ar[r]^{M(1_i)}  & M_i \\
      M_i \ar[u]^{\eta_i 1_{M_i}} \ar[ur]_{1_{M_i}}  & \\
     } 
    $$
That is,      
    $$
    \xymatrix @C=5pc {
      M_i(R(1_i)(x)) \ar[r]^-{M(1_i)_x}  & M_i(x) \\
      M_i(x) \ar[u]^{\eta_i 1_{M_i}x} \ar[ur]_{1_{M_i}(x)}  & \\
     } 
    $$
By the definition of $M_i$ in Construction 2, it is equivalent to check the following diagram commutes:  
    $$
    \xymatrix @C=9pc {
      F(_iR(1_i)(x)) \ar[r]^-{F[(\dots,0,1_{R(1_i)(x)},0.\cdots)]}  & F(_ix) \\
      F(_ix) \ar[u]^{F[(\dots,0,\eta_ix \circ \mu_{1_i,1_i}x,0.\cdots)]} \ar[ur]|{1_{F(_ix)}=F[(\dots,0,\eta_ix,0.\cdots)]}  & \\
     } 
    $$
where $\eta_ix \circ \mu_{1_i,1_i}x$ is the composite of
    $$
    \xymatrix @C=3pc {
      R(1_i)R(1_i)(x) \ar[r]^-{\mu_{1_i,1_i}x} & R(1_i)(x) \ar[r]^-{\eta_ix} & 1_{R(i)}(x)=x. \\
     } 
    $$
Since $F$ is a functor, it is equivalent to check the following diagram commutes:  
    $$
    \xymatrix @C=7pc {
      _iR(1_i)(x) \ar[d]_{(\dots,0,\eta_ix \circ \mu_{1_i,1_i}x,0.\cdots)}   & _ix \ar[dl]^{(\dots,0,\eta_ix,0.\cdots)} \ar[l]_-{(\dots,0,1_{R(1_i)(x)},0.\cdots)} \\
      _ix    & \\
     } 
    $$
This is true since we have the following commuting diagram:    
     $$
    \xymatrix @C=5pc {
   R(1_i)(x) \ar@/_2pc/[rrr]_-{\eta_ix} \ar[r]^-{\theta_{1_i,1_i}x} & R(1_i)R(1_i)(x) \ar[r]^-{R(1_i)(1_{R(1_i)(x)})} & R(1_i)R(1_i)(x) \ar[r]^-{\eta_ix \circ \mu_{1_i,1_i}x} & x. \\
     } 
    $$
Therefore, the condition (Mod.2) of the Definition \ref{rmod} is satisfied.    
This completes the proof.
\end{proof}

Now we can characterize $\rMod R$ as a functor category (see Theorem \ref{howeplus}). This can be viewed as the pseudofunctor analogue of the following Howe's result.

\begin{theorem} \label{How81} (see \cite[Proposition 5]{How81}, \cite[Theorem 5]{Mur06B} or \cite[Theorem 3.0.1]{Wu24})
Let $\calC$ be a small category and let $R:\calC^{\rm op} \to \mbox{\rm Ring}$ be a presheaf of unital rings on $\calC$ (note that $R$ is a \emph{strict} functor here!). Then we have the following equivalence 
	$$
	\rMod R \simeq (Gr(R)^{\op}, \Ab),
	$$
where $\rMod R$ is the category of right $R$-modules (see \cite[Definition 2.1.13]{Wu24}) and $Gr(R)$ is the linear Grothendieck construction of $R$ (see \cite[Definition 2.2.1]{Wu24}).
\end{theorem}

\begin{theorem} \label{howeplus}
Let $\calC$ be a small category and let $R: \calC \to \Add$ be a representation of the category $\calC$, then we have the following equivalence
 $$
 \rMod R \simeq (Gr(R)^{\op}, \Ab).
 $$    
\end{theorem}

\begin{proof}
Let $M \in \rMod R$ and $F \in (Gr(R)^{\op}, \Ab)$, we define
$$
 \Phi: \rMod R \to (Gr(R)^{\op}, \Ab)
$$ 
and
$$
 \Psi: (Gr(R)^{\op}, \Ab) \to \rMod R
$$ 
by $\Phi(M)=F_M$ and $\Psi(F)=M_F$ respectively. This two functors are well defined by Lemma \ref{F} and Lemma \ref{M}. We will check that they are inverse to each other.

Firstly, let's check that $\Psi \circ \Phi \cong \Id_{\rMod R}$. Since $(\Psi \circ \Phi)(M)=\Psi(\Phi(M))=\Psi(F_M)=M_{F_M}$, we have to show $M_{F_M} \cong M$ in $\rMod R$. In order to do that let's first prove $(M_{F_M})_i \cong M_i$ as functors in $(R(i)^{\op}, \Ab)$ for each $i \in \Ob \calC$. Let $g: x \to x'$ be a morphism in $R(i)$, then we have the following commuting diagram:
    $$
    \xymatrix{
     (M_{F_M})_i(x) \ar@{=}[d] \ar@{}[rrd]|{\mbox{Construction 2}}  &   &  (M_{F_M})_i(x') \ar[ll]_{(M_{F_M})_i(g)} \ar@{=}[d]  \\
    F_M(_ix) \ar@{=}[d]  &       & F_M(_ix') \ar[ll]|{F_M[(\cdots,0,g \circ \eta_ix, 0, \cdots)]} \ar@{=}[d] \\
     M_i(x) \ar@{}[rr]|{\mbox{Construction 1}} && M_i(x') \ar[dl]^{M_i(g \circ \eta_ix)} \\
     & M_i(R(1_i)x). \ar[ul]^{M(1_i)_x} & \\
     } 
    $$   
Therefore, $(M_{F_M})_i(g)=M(1_i)_x \circ M_i(g \circ \eta_ix)$. By (Mod.2) of Definition \ref{rmod}, one can deduce that $M(1_i)_x=M_i((\eta_ix)^{-1})$. Thus $(M_{F_M})_i(g)=M(1_i)_x \circ M_i(g \circ \eta_ix)=M_i((\eta_ix)^{-1}) \circ M_i(g \circ \eta_ix)= M_i(g)$. Hence $(M_{F_M})_i \cong M_i$ as functors in $(R(i)^{\op}, \Ab)$.

To show $M_{F_M} \cong M$ in $\rMod R$, we also have to prove that, for any $a: i \to j$,
    $$
    \xymatrix  @C=3pc {
   a^*(M_{F_M})_j \ar[r]^-{(M_{F_M})(a)_x} & (M_{F_M})_i & = & a^*M_j \ar[r]^-{M(a)_x} & M_i.  \\
     } 
    $$    
For each $x \in \Ob R(i)$, we have the following commuting diagram:
    $$
    \xymatrix{
     a^*(M_{F_M})_j(x) \ar[rr] \ar@{=}[d] \ar@{}[rrd]|{\mbox{Construction 2}}  &   &  (M_{F_M})_i(x)  \ar@{=}[d]  \\
    F_M(_jR(a)(x)) \ar@{=}[d] \ar[rr]|{F_M[(\cdots,0,1_{R(a)(x)}, 0, \cdots)]}  &       & F_M(_ix)  \ar@{=}[d] \\
     M_j(R(a)(x)) \ar@{}[rr]|{\mbox{Construction 1}} \ar[dr]_{M_j(1_{R(a)(x)})} && M_i(x)  \\
     & M_j(R(a)(x)). \ar[ur]_{M(a)_x} & \\
     } 
    $$
Therefore, we have
\begin{align*}
     &  (M_{F_M})(a)_x\\
	   = &  F_M[(\cdots,0,1_{R(a)(x)}, 0, \cdots)]\\
      = & M(a)_x \circ M_j(1_{R(a)(x)})\\
      M_j~preserves~identity= & M(a)_x \circ 1_{M_j(R(a)(x))} \\
      = & M(a)_x. \\
\end{align*}

Secondly, let's check that $\Phi \circ \Psi \cong \Id_{(Gr(R)^{\op,} \Ab)}$. Since $(\Phi \circ \Psi)(M)=\Phi(\Psi(M))=\Phi(M_F)=F_{M_F}$, we have to show $F_{M_F} \cong F$ in $(Gr(R)^{\op}, \Ab)$. For any morphism 
    $$
    \xymatrix @C=5pc {
    _ix \ar[r]^-{(\cdots,0,f_a,0,\cdots)} & _jy,
     } 
    $$
we have to show 
$$
F_{M_F}[(\cdots,0,f_a,0,\cdots)]=F[(\cdots,0,f_a,0,\cdots)].
$$
For $F_{M_F}[(\cdots,0,f_a,0,\cdots)]$, by Construction 2 and Construction 1, we have the following commuting diagram:
    $$
    \xymatrix @C=5pc {
      F_{M_F}(_ix) \ar@{=}[d] \ar[rr]^-{F_{M_F}[(\cdots,0,f_a,0,\cdots)]} \ar@{}[rrd]|{\mbox{Construction 1}} & & F_{M_F}(_jy) \ar@{=}[d] \\
     (M_F)_j(y) \ar[r]^-{(M_F)_j(f)} \ar@{}[rd]|{\mbox{Construction 2}} \ar@{=}[d]  & (M_F)_j(R(a)(x)) \ar@{}[rd]|{\mbox{Construction 2}} \ar[r]^-{M_F(a)_x} \ar@{=}[d] & (M_F)_i(x) \ar@{=}[d] \\
     F(_jy) \ar[r]_-{F[(\cdots,0,f \circ \mu_{1_j,a}x,0,\cdots)]} &  F(_jR(a)(x)) \ar[r]_-{F[(\cdots,0,1_{R(a)(x)},0,\cdots)]} & F(_ix). \\
     } 
    $$
It follows that
$$
F_{M_F}[(\cdots,0,f_a,0,\cdots)]=F[(\cdots,0,1_{R(a)(x)},0,\cdots)] \circ F[(\cdots,0,f \circ \mu_{1_j,a}x,0,\cdots)].
$$
In order to show 
$$
F_{M_F}[(\cdots,0,f_a,0,\cdots)]=F[(\cdots,0,f_a,0,\cdots)],
$$
it is enough to show
$$
(\cdots,0,f \circ \mu_{1_j,a}x,0,\cdots) \circ (\cdots,0,1_{R(a)(x)},0,\cdots) = (\cdots,0,f_a,0,\cdots).
$$
It is true because we have the following commuting diagram:
    $$
    \xymatrix @C=5pc {
     R(1_ja)(x) \ar[r]^{\theta_{1_j,a}x} \ar@{=}[d] & R(1_j)R(a)(x) \ar[r]^{R(1_j)1_{R(a)(x)}} & R(1_j)R(a)(x) \ar[d]^{f \circ \mu_{1_j,a}x} \\
     R(a)(x) \ar[rr]^{f} & & y. \\
     } 
    $$
This completes the proof.    

\end{proof}

\begin{example} \label{example1}
Let $k$ be a unital commutative ring, so it can be considered as a preadditive category with one object. Let $\calC$ be the path category of a quiver $Q$, and let $R$ be the constant functor, namely, having the constant value $k$ on objects and the constant value $id_k$ on morphisms. In this circumstance, the Theorem \ref{howeplus} above can be specialised to the well known equivalence between the category of the representations of $Q$ with values in $\rMod k$ and the category of the additive functors from the $k$-linear path category $Gr(R)$ of $Q$ to the Abelian group category $\Ab$.
\end{example}

\begin{corollary} \label{Groproj}
Let $\calC$ be a small category and let $R: \calC \to \Add$ be a representation of the category $\calC$. Then the category of right $R$-modules $\rMod R$ is a Grothendieck category and it has a projective generator.    
\end{corollary}

\begin{proof}
It follows from Theorem \ref{howeplus} immediately.    
\end{proof}

\begin{remark}
In \cite[Theorem 3.18]{EV17}, the authors proved the same result with different method, and our result is more general than theirs since we don't assume $\calC$ is a poset.  
\end{remark}

Similar to \cite[Theorem 4.0.1]{Wu24}, we can classify the hereditary torsion pairs in $\rMod R$ by linear Grothendieck topologies. For the definitions of hereditary torsion pairs and linear Grothendieck topologies, one can see \cite[Definition 2.3.1 \& 2.3.2 and Definition 2.1.15]{Wu24}.   

\begin{corollary} \label{htp}   
Let $\calC$ be a small category and let $R: \calC \to \Add$ be a representation of the category $\calC$. Then there is an (explicit) one-to-one correspondence between linear Grothendieck topologies on $Gr(R)$ and hereditary torsion pairs in $\rMod R$. 
\end{corollary}

\begin{proof}
It follows from Theorem \ref{howeplus} and \cite[Theorem 3.7]{PSV21}.
\end{proof}

\subsection{Characterizing it as module categories of algebras} \label{chaalg}
In this subsection, we will characterize the category of right $R$-modules $\rMod R$ as the category of modules over a pseudoskew category algebra $R[\calC]$. This can be seen as a higher analogue of the \cite[Theorem A]{WX23}.
   
\begin{theorem} \label{higherskew}
Let $\calC$ be a small category and let $R: \calC \to \Add$ be a representation of the category $\calC$. If $\Ob \calC < +\infty$ and $\Ob R(i) < +\infty$ for all $i \in \Ob \calC$, then we have the following equivalence
 $$
 \rMod R \simeq \rMod R[\calC].
 $$    
\end{theorem}

\begin{proof}
By Theorem \ref{howeplus}, we know that 
$$
 \rMod R \simeq (Gr(R)^{\op}, \Ab).
$$
Let 
$$
G:=\bigoplus_{_ix \in \Ob Gr(R)}\Hom_{Gr(R)}(-,_ix),
$$
then $G$ is a projective generator of $(Gr(R)^{\op}, \Ab)$, hence of $\rMod R$. Since $\Ob \calC < +\infty$ and $\Ob R(i) < +\infty$ for all $i \in \Ob \calC$, so $\Ob Gr(R) < +\infty$, it follows that $G$ is even small. Then we will immediately have the following equivalence
$$
 \rMod R \simeq \rMod \End(G).
$$

Now, let's compute $\End(G)$.

\begin{align*}
     & \End(G) \\
	   = & \Hom\left(\bigoplus_{_ix \in \Ob Gr(R)}\Hom_{Gr(R)}(-,_ix), \bigoplus_{_jy \in \Ob Gr(R)}\Hom_{Gr(R)}(-,_jy)\right)\\
	compa.\ of\ 1st\ argu. \cong & \bigoplus_{_jy \in \Ob Gr(R)} \Hom\left(\bigoplus_{_ix \in \Ob Gr(R)}\Hom_{Gr(R)}(-,_ix), \Hom_{Gr(R)}(-,_jy)\right)\\
     \Ob Gr(R) < +\infty  \cong & \bigoplus_{_jy \in \Ob Gr(R)} \bigoplus_{_ix \in \Ob Gr(R)}  \Hom\left(\Hom_{Gr(R)}(-,_ix), \Hom_{Gr(R)}(-,_jy)\right)\\
     Yoneda \cong & \bigoplus_{_jy \in \Ob Gr(R)} \bigoplus_{_ix \in \Ob Gr(R)}\Hom_{Gr(R)}(_ix, _jy)\\
    \left(  (f_a)_a \mapsto \sum_a f_a \right) \cong & R[\calC].  
\end{align*}

Therefore, we have 
$$
 \rMod R \simeq \rMod \End(G) \simeq \rMod R[\calC].
$$

This completes the proof.
\end{proof}

\begin{example}
Let $k$, $\calC$, $Q$, $R$ as that of Example \ref{example1} with $\Ob \calC < +\infty$, then the Theorem \ref{higherskew} above can be specialised to the well known equivalence between the category of the representations of $Q$ with values in $\rMod k$ and the category of the modules over the path algebra $kQ$.
\end{example}

\section*{Acknowledgments}
I would like to thank my Ph.D. supervisor Prof. Fei Xu \begin{CJK*}{UTF8}{}
\CJKtilde \CJKfamily{gbsn}(徐斐) \end{CJK*} in Shantou University for motivating me to think representation theory higher categorically. I also want to thank my girlfriend Wenwen Sun \begin{CJK*}{UTF8}{} \CJKtilde \CJKfamily{gbsn}($\heartsuit$~孙雯雯~$\heartsuit$) \end{CJK*} for her love. 
\bibliographystyle{plain}
\bibliography{ref}
\end{document}